\documentclass[12 pt]{amsart}

\newtheorem{theorem}{Theorem}[section]
\newtheorem{lemma}[theorem]{Lemma}

\theoremstyle{definition}

\newtheorem{remark}[theorem]{Remark}

\newcommand{\Aut}{\operatorname{Aut}}

\theoremstyle{remark}
\usepackage{amscd,amssymb}
\begin{document}
\baselineskip 17 pt

\title[Stable tameness of automorphisms of $F\langle x,y,z\rangle$ fixing $z$]
{Stable tameness of automorphisms of $F\langle x,y,z\rangle$ fixing $z$}

\author[A.Belov-Kanel and Jie-Tai Yu]
{Alexei Belov-Kanel and Jie-Tai Yu}
\address{Department of Mathematics,
Bar-Ilan University Ramat-Gan, 52900 Israel} \email{beloval@cs.biu.ac.il,\ kanelster@gmail.com}
\address{Department of Mathematics, The University of Hong
Kong, Hong Kong SAR, China} \email{yujt@hkusua.hku.hk,\
yujietai@yahoo.com}

%\thanks
%{The research of Alexei Belov-Kanel  was partially supported by an NSFC Grant.}

\thanks{The research of Jie-Tai Yu was partially
supported by an RGC-GRF Grant.}

\subjclass[2000] {Primary 13S10, 16S10. Secondary 13F20, 13W20,
14R10, 16W20, 16Z05.}

\keywords{Automorphisms, coordinates,
 polynomial
algebras,  free associative algebras, stably tameness,
lifting problem.}

\maketitle

\noindent $\mathbf{Abstract.}$\ It is proved that
every $z$-automorphism ($z$-coordinates, respectively) of the free associative algebra $F\langle x,y,z\rangle$ over an arbitrary field $F$ is stably tame.

%It is also proved
%that every $z$-automorphism of $F\langle x,y,z\rangle$ is stably tame.

\smallskip

\section{\bf Introduction and main results}

\smallskip

\noindent An $F$-automorphism of a free associative algebra
$F\langle x_1,\dots,x_n\rangle$ (a polynomial algebra $F[x_1,\dots,x_n]$)
is {\it elementary} if it fixes all variables except one. An $F$-automorphism
is {\it tame} if it is product of elementary automorphisms. An $F$-automorphisms
$(f_1,\dots,f_n)$ of $F\langle x_1,\dots,x_n\rangle$ ($F[x_1,\dots,x_n]$)
is {\it stably tame} if there exists a nonnegative integer $m$,
the automorphism $(f_1,\dots,f_n,x_{n+1},\dots,x_{n+m})$
of $F\langle x_1,\dots,x_{n+m}\rangle$\

\noindent ($F[x_1,\dots,x_{n+m}]$) is tame.

\

\noindent From now on all automorphisms are $F$-automorphisms.

\

\noindent Whether every automorphism of a free associative algebra
(polynomial algebra) is stably tame, is a long-standing and interesting
open question.

\

\noindent In \cite{S}, stable tameness of some special kind of automorphisms
of polynomial and free associative algebras were obtained. In \cite{BEW},
it is proved that every fixing $z$ automorphism of the polynomial algebra
$F[x,y,z]$ over a field $F$ of characteristic zero is stably tame, among other things. It is the first big step
for attacking the stably tameness problem.

\

\noindent In  this paper, based our recent result of the lifting problem,
we prove the following

\begin{theorem} \label{stablytame}
Every fixing $z$ automorphism of the free associative algebra
$F\langle x,y,z\rangle$ over an arbitrary field $F$ is
stably tame and becomes tame after adding one variable.
\end{theorem}

\noindent A polynomial $f\in F\langle x_1,\dots,x_n\rangle$ is a {\it coordinate}
if $(f,f_2,\dots,f_n)$ is an automorphism for some $f_2,\dots,f_n\in F\langle x_1,\dots,x_n\rangle$.
The coordinate $f$ is stably tame if $(f,f_2,\dots,f_n)$ is stably tame. A coordinate $f$
is a {\it $x_n$-coordinate} if there exists an automorphism
$(f,f_2,\dots,f_{n-1},x_n)$.

\noindent As a direct consequence of Theorem \ref{stablytame}, we obtain

\begin{theorem} Every $z$-coordinate of $F\langle x,y,z\rangle$ is stably tame.
\end{theorem}

\

\section{\bf Proofs}

\noindent To prove Theorem \ref{stablytame}, we only need
to prove the following

\begin{theorem} \label{coproduct}
For every automorphism $(f,g)$ in
$\Aut_{F[z]}F\langle x,y,z\rangle$,
$(f,g,t)$ is a tame automorphism in
$\Aut_{F[z]}F\langle x,y,z,t\rangle$.
\end{theorem}

\noindent To prove Theorem \ref{coproduct}, we need some preliminaries.

\begin{theorem}\label{decomposation}
An automorphism $(f,\ g)$ in $\Aut_{F[z]} F\langle x,\ y,\ z\rangle$,
can be canonically decomposed as product of the following type
of automorphisms:
i) Linear automorphisms in $\Aut_{F[z]}F\langle x,y,z\rangle$;\
ii) Automorphisms which can be obtained by an
elementary automorphism in
$$\Aut_{F[z]}F\langle x,y,z\rangle,$$ conjugated by a linear automorphism
in $$\Aut_{F(z)}F(z)*_FF\langle x,y\rangle$$
\end{theorem}
\begin{proof} It is Theorem 3.4 in \cite{BKY}.
\end{proof}

\begin{lemma}          \label{LmCoprdPrdIntsct}
Suppose a polynomial $f\in F\langle x,y,z\rangle$ is neither a left multiple nor a right multiple of any nontrivial polynomial in $F[z]\backslash F$. Then $(F(z)*F[f])\cap(F\langle x,y,z\rangle)=F(z)*F[f]$.
\end{lemma}

\begin{proof}
Suppose $R\in F(z)*F[f]\backslash F\langle f,z\rangle =(F(z)*F[f])\backslash F$. We need to prove that
$R\in F(z)*F\langle x,y\rangle\backslash F\langle\ x,y,z\rangle$, i.e.  $R\notin  F\langle x,y,z\rangle$.

\

\noindent Let $\{M_i\}_{i\in I}$ be an arbitrary $F$-basis of $F(z)*F[f]$. As $R\in F(z)*F[f]$, it can be expressed
uniquely $$R=\sum_i \alpha_ifM_i,$$ where  $\alpha_i\in F(z)$.

\

\noindent It is easy to see that if $\alpha_k\in F(z)\backslash F[z]$ for some $k$,\ then $R\notin F\langle x,y,z\rangle$. Hence we may assume that
$\alpha_k\in F[z]$ for all $k$. Then $R$ has the  form $R=\sum_{k=0}^n\gamma_kz^kfN_k$,\ $\gamma_k\in F$. The set $\{z^k\}_{k=0}^\infty$ can be completed to a basis $\{e_i\}$ of $F(z)$ as a vector space over $F$, and every element $x$ in the coproduct $F(z)*F[f]$ can be expressed as $x=e_ifh_i$ in canonical way. In our situation $h_i\ne 0$ only for $e_i\in \{z^k\}_{k=0}^\infty$. Hence all $N_k\in F(z)*F[f]$. On the other hand, if some of $N_k\notin F\langle x,y,z\rangle$, then $fN_k\notin F\langle x,y,z\rangle$, because $f$ is not right-divisible  by any polynomial from $F[z]\backslash F$. In this situation $R=\sum_{k=0}^n\gamma_kz^kfN_k\notin F\langle x,y,z\rangle$. Hence for all $k$
$$(N_k\in F(z)*F[f])\cap(F\langle x,y,z\rangle)=F(z)*F[f]$$
and $\deg(N_k)<\deg(R)$. We conclude by induction on the degree of $R$.
\end{proof}

\noindent Let $z_l$ denote the left multiplication operator on $z$,\ $z_r$ the
right multiplication operator. An automorphism $\psi$ in $\Aut_{F[z]}F\langle x,y,z\rangle$
linear in both $x$ and $y$ has the following form: $\psi: x\to
a_{11}x+a_{12}y, y\to a_{21}x+a_{22}y$ where $a_{ij}\in
F[z_l,z_r]$. It should be pointed out, the study of such automorphisms is equivalent to
the study of invertible $2\times 2$ matrices over the polynomial ring of
two commuting variables $z_l$ and $z_r$ over a field.

\begin{lemma}         \label{noncommcon}
Let  $\psi: x\to a_{11}x+a_{12}y, y\to
a_{21}x+a_{22}y$ be a linear automorphism of $F(z_l,z_r)\langle x,y\rangle$ where $a_{ij}\in F(z_l,z_r)$ and let
$\varphi: x\to x, y+Q(x)$ be an elementary $z$-automorphism
of $F\langle x,y,z\rangle$. Suppose
$\phi=\psi\circ\varphi\circ\psi^{-1}\in\Aut_{F[z]}F\langle x,y,z\rangle$. Then it has
the following form
$$
\phi:\ x\to x+ bh(ax+by),\ y\to y-ah(ax+by);\quad a, b
\in F[z_l,\ z_r],
$$
where $h(t)\in F\langle z,t\rangle$.
\end{lemma}

\begin{proof}
Let $\alpha\in F[z_l,\ z_r]$ be the least common multiple of the
denominators of $a_{21}, a_{22}$, $a=a_{21}\alpha,
b=a_{22}\alpha$. Then
$$
\phi:\ x\to x+ bQ((ax+by)/\alpha),\ y\to y-aQ((ax+by)/\alpha);$$
where $\quad a,\ b\in F[z_l,\ z_r]$.

\

\noindent As $a, b$ are relatively prime,\ $Q((ax+by)/\alpha)\in F\langle x,y,z\rangle$. Hence by Lemma \ref{LmCoprdPrdIntsct}, the coefficients of $Q((ax+by)/\alpha)\in F\langle z,ax+by\rangle$.
Therefore $Q((ax+by)/\alpha)$ must  have the  form $h(ax+by)$ for some $h(t)\in F\langle z,t\rangle$.
\end{proof}

\begin{lemma}         \label{conjstablytame}
Let $\psi$ be a $z$-automorphism of $F\langle x,y,z\rangle$
in the form
$$
x\to x+ bh(ax+by),\ y\to y-ah(ax+by);\quad a, b \in
F[z_l,z_r]
$$
for some polynomial $h(t)\in F\langle z,t\rangle$.
Then it is stably tame and becomes tame after adding one variable.
\end{lemma}
\begin{proof}
Based on the method of Smith \cite{S},
$$
(x+bh(ax+by),y-ah(ax+by),t)=$$
$$(x,y,t-h(ax+by))(x-bt,y+at,t)(x,y,z, t+h(ax+by))(x+bt,y-at,t).$$
\end{proof}

\begin{remark}
The general Anick type automorphisms  (see \cite{DY1, DY2})
$$(x+zg(xz-zy, z),
y+g(xz-zy, z)z),$$ in $\Aut_{F[z]}F\langle x,\ y,\ z\rangle$,
with an arbitrary polynimial $g(t,\ s)\in F\langle t,\ s\rangle$,
are obviously covered by Lemma \ref{conjstablytame}.
\end{remark}

\noindent {\bf Proof of Theorem \ref{stablytame}.}
First, an automorpphism of Type i) in Theorem \ref{coproduct} is stably tame and becomes tame after adding one variable,
according to \cite{DY}.

\

\noindent Now, suppose we have an automorphism of  Type ii) in Theorem \ref{coproduct},
we are done by Lemma \ref{conjstablytame}.\qed

\

\section{\bf Acknowledgements}

\

\noindent Jie-Tai Yu would like to thank David Wright for sending him
an early version of \cite{BEW} in July 2007.
The authors also thank Vesselin Drensky and Leonid Makar-Limanov for
comments and remarks.

\

\end{document}